\documentclass{amsart}
\usepackage{amssymb}
\usepackage[all]{xy}
\usepackage{ifthen}

\newtheorem{thm}{Theorem}
\newtheorem{lemma}[thm]{Lemma}
\newtheorem{cor}[thm]{Corollary}
\newtheorem{propn}[thm]{Proposition}
\newtheorem{defn}[thm]{Definition}

\theoremstyle{remark}

\newtheorem{remark}[thm]{Remark}
\newtheorem{example}[thm]{Example}
\newtheorem{discussion}[thm]{Discussion}

\DeclareMathOperator{\height}{ht}
\DeclareMathOperator{\Spec}{Spec}
\DeclareMathOperator{\depth}{depth}
\newcommand{\defeq}{:=}
\DeclareMathOperator{\tor}{Tor}
\DeclareMathOperator{\unm}{Unm}
\DeclareMathOperator{\ass}{Ass}
\newcommand{\naturals}{\mathbb{N}}
\newcommand{\ints}{\mathbb{Z}}
\newcommand{\reals}{\mathbb{R}}
\DeclareMathOperator{\rhomo}{\widetilde H}
\DeclareMathOperator{\reg}{reg}
\DeclareMathOperator{\projdim}{pd}
\DeclareMathOperator{\arank}{ara}
\DeclareMathOperator{\image}{Im}
\newcommand{\serreS}[1]{\ensuremath{(S_{#1})\,}}

\makeatletter

\@addtoreset{thm}{section}
\def\thethm{\thesection.\@arabic\c@thm}
\def\theenumi{\@alph\c@enumi}
\def\@lbibitem[#1]#2{\def\@biblab@l{#1}
      \ifthenelse{\equal{#2}{EGA}}{\def \@biblab@l{EGA}}{}
      \ifthenelse{\equal{#2}{M2}}{\def \@biblab@l{\texttt{M2}}}{}
      \item[\@biblabel{\@biblab@l}\hfill]\if@filesw
      {\let\protect\noexpand
       \immediate
       \write\@auxout{\string\bibcite{#2}{\@biblab@l}}}\fi\ignorespaces}

\makeatother

\title{Regularity, Depth and Arithmetic Rank of Bipartite Edge Ideals}
\author{Manoj Kummini}
\address{Purdue University\\Lafayette, IN 47907, USA.}
\email{nkummini@math.purdue.edu}

\begin{document}

\begin{abstract}
We study minimal free resolutions of edge ideals of bipartite graphs.  We
associate a directed graph to a bipartite graph whose edge ideal is
unmixed, and give expressions for the regularity and the depth of the edge
ideal in terms of invariants of the directed graph. For some classes of
unmixed edge ideals, we show that the arithmetic rank of the ideal equals
projective dimension of its quotient.
\end{abstract}

\maketitle

\section{Introduction}
\label{sec:intro}

Let $G$ be a simple graph on a finite vertex set $V$
without any isolated vertices. Let $\Bbbk$ be a
field.  Set $R = \Bbbk[V]$, treating the elements of $V$ as indeterminates.
Let $I$ be the \emph{edge ideal} of $G$ in $R$, \textit{i.e.}, the ideal
generated by the square-free quadratic monomials $xy$, where $x, y \in V$
and there is an edge between $x$ and $y$ in $G$. In this paper, we study
(Castelnuovo-Mumford) regularity, depth and arithmetic rank of edge ideals
of bipartite graphs. Recall that $G$ is said to be \emph{bipartite} if
there exists a partition $V = V_1 \bigsqcup V_2$ such that every edge in
$G$ is of the form $xy$ with $x \in V_1$ and $y \in V_2$. 

When $I$ is unmixed (more generally, when $G$ has a perfect matching ---
see Section~\ref{sec:edgeIdeals} for details), we have that $|V_1| = |V_2|
= \height I$. To such a bipartite graph, we associate a directed graph
$\mathfrak d_G$ on the vertex set $\{1, \ldots, \height I\}$. This is
motivated by a paper of J.~Herzog and T.~Hibi~\cite{HeHiCMbip05} which
studies a similar association between posets and bipartite graphs with
Cohen-Macaulay edge ideals. Using this, we show that 
\def\unmBipRegIsrIThm{Let $G$ be an unmixed bipartite graph
with edge ideal $I$. Then $\reg R/I = \max\{|A| : A \ \text{is an antichain
in}\ \mathfrak{d}_G\}$. In particular, $\reg R/I$ is the maximum size of a
pairwise disconnected set of edges in $G$.} 
\begin{thm}
\label{thm:unmBipRegIsrI}
%Let $G$ be an unmixed bipartite graph with edge ideal $I$. Then $\reg R/I =
%\max\{|A| : A \in \mathcal A_{\mathfrak{d}_G}\}$. In particular, $\reg R/I
%= r(I)$.
\unmBipRegIsrIThm
\end{thm}
We say that $G$ is \emph{unmixed} (respectively, \emph{Cohen-Macaulay}) if
$R/I$ is unmixed (respectively, Cohen-Macaulay).  The notion of pairwise
disconnected sets of edges in graphs was introduced by
X.~Zheng~\cite{Zhereslnfacets04} who showed that if $I$ is the edge ideal
of a tree (an acyclic graph) then $\reg R/I$ is the maximum size of a
pairwise disconnected set of edges~\cite[Theorem~2.18]{Zhereslnfacets04}.
Additionally, see~\cite[Corollary~6.9]{HvThypergraphs08}, for the same
conclusion for the edge ideals of chordal graphs.  For arbitrary graphs,
the maximum size of a pairwise disconnected set of edges is a lower bound
for $\reg R/I$; this follows essentially
from~\cite[Lemma~2.2]{KatzmanCharIndep06}. 

A \emph{strong component} of a directed graph is a set of vertices maximal
with the property that for every $i, j$ in the set, there is a directed
path from $i$ to $j$. The following statement about depth, which follows
from Corollary~\ref{thm:depthBdInTheZeta}, has also been observed
by C.~Huneke and M.~Katzman:
\def\unmBipDepthWeakVersion{Let $G$ be an unmixed bipartite graph, with
edge ideal $I$ and associated directed graph $\mathfrak d_G$. If $\mathfrak
d_G$ has $t$ strong components, then $\depth R/I \geq t$.}
\begin{thm}
\label{thm:unmBipDepthWeakVersion}
\unmBipDepthWeakVersion
\end{thm}

The problem of determining the minimum number of equations required to
generate a monomial ideal up to radical (called the \emph{arithmetic rank}
of the ideal) was first studied by
P.~Schenzel and W.~Vogel~\cite{ScVoSTCI77}, T.~Schmitt and
Vogel~\cite{ScmVoSTCI79} and G.~Lyubeznik~\cite{LyubArithRk88}.
Lyubeznik showed that for a square-free monomial ideal $I$, $\arank I \geq
\projdim R/I$~\cite[Proposition ~3]{LyubArithRk88}. Upper bounds for arithmetic rank have also been considered
by M.~Barile~\cite{BariNoEqnsCertain96}
and~\cite{BariMonomialIdealsNote06}, building on the work of Schmitt and
Vogel mentioned above. In~\cite{KTYsmallADeg08}, K.~Kimura,
N.~Terai and K.-i.~Yoshida raise the question of equality of $\arank I$ and
$\projdim R/I$, and answer it in some
cases~\cite[Theorem~1.1]{KTYsmallADeg08}. It is known, however, due to
Z.~Yan~\cite[Example~2]{YanEtaleGorMcp00} that, in general, $\projdim R/I$
and $\arank I$ need not be equal. 

If $G$ is an unmixed bipartite graph, then we can construct a maximal
subgraph $\breve G$ which is Cohen-Macaulay; this corresponds to taking a
maximal directed acyclic subgraph of $\mathfrak d_G$. If $G$ is, further,
Cohen-Macaulay, then $\breve G = G$. Let $\breve I$ be the edge ideal of
$\breve G$. We show that
\def\thmCMBipSTCI{Let $G$ be an unmixed bipartite graph with edge
ideal $I$. Then $\arank I \leq \arank \breve I + \projdim R/I - \height I$.
If further a maximal acyclic subgraph of $\mathfrak d_G$ can be embedded in
$\naturals^2$, then $\arank I = \projdim R/I$.}
\begin{thm}
\label{thm:compatLinznImpliesSTCI}
\thmCMBipSTCI
\end{thm}
Thus, if $G$ is Cohen-Macaulay and $\mathfrak d_G$ has an embedding in
$\naturals^2$, then $R/I$ is a set-theoretic complete intersection,
\textit{i.e.}, it can be defined by $\height I$ equations.

The next section contains definitions, notation and some preliminary
observations. Theorems~\ref{thm:unmBipRegIsrI}
and~\ref{thm:unmBipDepthWeakVersion} are proved in
Section~\ref{sec:regDepthBipGr}. A proof of
Theorem~\ref{thm:compatLinznImpliesSTCI} is presented in
Section~\ref{sec:arithRk}.

\section{Edge Ideals}
\label{sec:edgeIdeals}

We fix the following notation: $\Bbbk$ is a field, $V$ is a
finite set of indeterminates over $\Bbbk$, $G$ is a simple graph on $V$ 
without any isolated vertices and
$R = \Bbbk[V]$ is a polynomial ring. We take
$I \subseteq R$ to be a square-free monomial ideal; later, we will assume
that $I$ is the edge ideal of $G$. Set $c \defeq \height I$. References for
homological aspects of monomial ideals, for graph theory and for results on
posets,
respectively, are~\cite[Part I]{MiStCCA05},~\cite{WestGraphTheory96}
and~\cite[Chapter~3]{StanEC1}. We will use ``multigraded'' and
``multidegree'' to refer to the grading of $R$ by $\naturals^{|V|}$ and the
degrees in this grading.

The \emph{multigraded Betti numbers} of $R/I$ are $\beta_{l,\sigma}(R/I)
\defeq \dim_\Bbbk \tor_l^R(\Bbbk,R/I)_\sigma$. For $j \in \ints$, the
($\naturals$-graded) \emph{Betti numbers} are $\beta_{l,j} \defeq
\dim_\Bbbk \tor_l^R(\Bbbk,R/I)_j$. We note that $\beta_{l,j}(\cdot) =
\sum\beta_{l, \sigma}(\cdot)$, where the sum is taken over the set of
$\sigma$ with $|\sigma|=j$. (Here $|.|$ denotes the total degree of a
multidegree.) To represent a multidegree, we will often use the unique
monomial in $R$ of that multidegree; further, if that monomial is
square-free, we will use its support, \textit{i.e.}, the set of
variables dividing it. 

Let $\Delta$ be the Stanley-Reisner complex of $I$. The correspondence
between non-faces of $\Delta$ and monomials in $I$ can also be expressed as
follows: for any monomial prime ideal $\mathfrak p \in \Spec R$, $I
\subseteq \mathfrak p$ if and only if $\mathfrak p = (\bar F)R$, the ideal
generated by $\bar F \defeq V \setminus F$, for some $F \in
\Delta$~\cite[Theorem~1.7]{MiStCCA05}. Thus minimal prime ideals of $R/I$
correspond to complements of maximal faces of $\Delta$. The \emph{Alexander
dual} of $\Delta$, denoted $\Delta^\star$, is the simplicial complex
$\{\bar F : F \not \in \Delta\}$.  Let $m \in \naturals$ and $F_i \subseteq
V, 1 \leq i \leq m$ be such that $\prod_{x \in F_i} x, 1 \leq i \leq m$ are
the minimal monomial generators of $I$. The \emph{Alexander dual} of $I$,
denoted $I^\star$, is the square-free monomial ideal $\cap_{i=1}^m (F_i)$.
If $I$ is the Stanley-Reisner ideal of $\Delta$, then $\bar F_i, 1 \leq i
\leq m$ are precisely the facets of $\Delta^\star$. Hence $I^\star$ is the
Stanley-Reisner ideal of $\Delta^\star$. We will need the following theorem
of Terai:
\begin{propn}
[Terai~\protect{\cite{Terai99regdual}};~\protect{\cite[Theorem
5.59]{MiStCCA05}}] 
\label{thm:TeraiPdimReg}
For any square-free monomial ideal $J$, $\projdim R/J = \reg J^\star$.
\hfill\qedsymbol
\end{propn}

The relation between simplicial homology and multigraded Betti numbers is
given by Hochster's Formula~\cite[Corollary~5.12 and
Corollary~1.40]{MiStCCA05}. For $\sigma \subseteq V$, we denote by
$\Delta|_\sigma$ the simplicial complex obtained by taking all the faces of
$\Delta$ whose vertices belong to $\sigma$. Note that $\Delta|_\sigma$ is
the Stanley-Reisner complex of the ideal $I \cap \Bbbk[\sigma]$. Similarly,
define the \emph{link}, $\mathrm{lk}_\Delta(\sigma)$, of $\sigma$ in
$\Delta$ to be the simplicial complex $\{F \setminus \sigma: F \in \Delta,
\sigma \subseteq F\}$. Its Stanley-Reisner ideal in $\Bbbk[\bar \sigma]$ is
$(I: \sigma) \cap \Bbbk[\bar \sigma]$. First, the multidegrees $\sigma$
with $\beta_{l,\sigma}(R/I) \neq 0$ are square-free. Secondly, for all
square-free multidegrees $\sigma$, 
\begin{align}
\beta_{l,\sigma} (R/I) & = \dim_\Bbbk \rhomo_{\vert \sigma \vert - l -
1}(\Delta \vert_\sigma; \Bbbk), \quad \text{and}
\label{eqn:hochsterFormulaDirec}\\
\beta_{l,\sigma}(I^\star) & = \dim_\Bbbk
\rhomo_{l-1}(\mathrm{lk}_\Delta(\bar \sigma); \Bbbk). \nonumber
%\label{eqn:hochsterFormulaDual}
\end{align}
Combining these two formulas we see that
\begin{equation}
\label{eqn:alexDualBettiNos}
\beta_{l,\sigma}(I^\star) =
\beta_{|\sigma|-l, \sigma}\left(\frac{R}{(I:\bar \sigma)}\right).
\end{equation}
We add, parenthetically, that links of faces in Cohen-Macaulay complexes
are themselves Cohen-Macaulay.

We now describe how the graded Betti
numbers change under restriction to a subset of the variables and under
taking colons.
\begin{lemma}
\label{thm:varSetRestrictColon}
Let $I \subseteq R = \Bbbk[V]$ be a square-free monomial ideal, $x \in
V$, $l, j \in \naturals$ and $\sigma \subseteq V$ with $|\sigma| = j$.
\begin{enumerate}

\item \label{thm:varSetRestrict} Let $W \subseteq V$ and $J = (I \cap
\Bbbk[W])R$. Then,
\[
\beta_{l,\sigma}(R/J) = 
\begin{cases}
0, & \sigma \nsubseteq W, \\
\beta_{l,\sigma}(R/I), & \sigma \subseteq W.
\end{cases}
\]
In particular, $\beta_{l,j}(R/J) \leq \beta_{l,j}(R/I)$.

\item \label{thm:varColon} If $\beta_{l,\sigma}(R/(I:x)) \neq 0$, then
$\beta_{l,\sigma}(R/I) \neq 0$ or $\beta_{l,\sigma\cup\{x\}}(R/I) \neq
0$.
\end{enumerate}
\end{lemma}

\begin{proof}
\eqref{thm:varSetRestrict}: The second assertion follows from the first,
which we now prove. Let $\tilde \Delta$ be the Stanley-Reisner complex of
$J$.  Since for all $x \in V \setminus W$, $x$ does not belong to any
minimal prime ideal of $R/J$, we see that every maximal face of $\tilde
\Delta$ is contains $V \setminus W$. Hence if $\sigma \not \subseteq W$,
then for all $x \in \sigma \setminus W$, $\tilde \Delta |_\sigma$ is a cone
with vertex $x$, which, being contractible, has zero reduced homology.
Applying~\eqref{eqn:hochsterFormulaDirec}, we see that
$\beta_{l,\sigma}(R/J) = 0$.

Now let $\sigma \subseteq W$ and $F \subseteq V$. Then $F \in \Delta
\vert_\sigma$ if and only if $I \subseteq (\bar F)R$ and $F \subseteq
\sigma$, which holds if and only if $J \subseteq (\bar F)R$ and $F
\subseteq \sigma$, which, in turn, holds if
and only if $F \in \tilde \Delta \vert_\sigma$.
Apply~\eqref{eqn:hochsterFormulaDirec} again to get \[
\beta_{l,\sigma} (R/J) = 
\rhomo_{\vert \sigma \vert - l - 1}(\tilde \Delta \vert_\sigma; \Bbbk)
= \rhomo_{\vert \sigma \vert - l - 1}(\Delta \vert_\sigma; \Bbbk)
= \beta_{l,\sigma} (R/I).
\]

\eqref{thm:varColon}: We take the multigraded exact sequence of
$R$-modules:
\begin{equation}
\label{seq:fundColSpecialMulti}
\xymatrix{%
0 \ar[r] & \frac{R}{(I:x)}(-x) \ar[r] & \frac{R}I \ar[r] &
\frac{R}{(I,x)} \ar[r] & 0.}
\end{equation}
The corresponding multigraded long exact sequence of $\tor$ is
\[
\xymatrix{%
\cdots \ar[r] & \tor_{l+1}(\Bbbk, \frac{R}{(I,x)}) \ar[r] &
\tor_l(\Bbbk, \frac{R}{(I:x)}(-x)) \ar[r] & 
\tor_l(\Bbbk, \frac{R}I) \ar[r]
& \cdots.}
\]

Let $W = V \setminus \{x\}$ and $J = (I \cap \Bbbk[W])R$. Since
$\beta_{l,\sigma}(R/(I:x)) \neq 0$ and $x$ does not divide any monomial
minimal generator of $(I:x)$, we have, by the same argument as in
\eqref{thm:varSetRestrict}, $\sigma \subseteq W$. Let $\tau = \sigma
\cup \{x\}$. First observe that 
\[
\tor_l\left(\Bbbk, \frac{R}{(I:x)}\right)_\sigma \simeq 
\tor_l\left(\Bbbk, \frac{R}{(I:x)}(-x)\right)_\tau.
\]
Let us assume that $\beta_{l,\tau}(R/I) = 0$, because, if
$\beta_{l,\tau}(R/I) \neq 0$, there is nothing to prove. Then, the above
long exact sequence of $\tor$, restricted to the multidegree $\tau$,
implies that $\tor_{l+1}(\Bbbk, \frac{R}{(I,x)})_\tau \neq 0$. Now, since
$(I,x) = (J,x)$, we see further $\tor_{l+1}(\Bbbk, \frac{R}{(J,x)})_\tau
\neq 0$. 

Since $x$ is a non-zerodivisor on $R/J$, we have a multigraded short
exact sequence
\[
\xymatrix{%
0 \ar[r] & \frac{R}J(-x) \ar[r] & \frac{R}J \ar[r] &
\frac{R}{(J,x)} \ar[r] & 0,}
\]
which gives the following long exact sequence of
$\tor$:
\[
\xymatrix{%
\cdots \ar[r] & 
\tor_{l+1}(\Bbbk, \frac{R}J) \ar[r] &
\tor_{l+1}(\Bbbk, \frac{R}{(J,x)}) \ar[r] &
\tor_l(\Bbbk, \frac{R}J(-x)) \ar[r] & \cdots.}
\]

Since $x$ does not divide any minimal monomial generator of $J$,
$\beta_{l+1, \tau}(R/J) = 0$. Therefore $\tor_l(\Bbbk, \frac{R}J(-x))_\tau
\neq 0$, or, equivalently, $\tor_l(\Bbbk, \frac{R}J)_\sigma \neq 0$.
By~\eqref{thm:varSetRestrict} above, $\beta_{l,\sigma} (R/I) \neq 0$.
\end{proof}

If $\mathfrak p \subseteq R$ is a prime ideal such that $\height \mathfrak
p = c = \height I$ and $I \subseteq \mathfrak p$, then we say that
$\mathfrak p$ is an \emph{unmixed} associated prime ideal of $R/I$. Denote
the set of unmixed associated prime ideals of $R/I$ by $\unm R/I $. Unmixed
prime ideals are necessarily minimal over $I$, so $\unm R/I  \subseteq \ass
R/I$; we say that $I$ or $R/I$ is \emph{unmixed} if $\unm R/I  = \ass R/I$. 

We now restrict our attention to edge ideals of graphs. Every square-free
quadratic monomial ideal can be considered as the edge ideal of some
simple graph. The theory of edge ideals is systematically developed
in~\cite[Chapter~6]{VillmonAlg01}. Hereafter $I$ is the edge ideal of $G$,
which we have set to be a simple graph on $V$. A \emph{vertex cover} of $G$
is a set $A \subseteq V$ such that whenever $xy$ is an edge of $G$, $x \in
A$ or $y \in A$. It is easy to see that for all $A \subseteq V$, $A$ is a
vertex cover of $G$ if and only if the prime ideal $(x : x \in A)$ contains
$I$. Since $I$ is square-free, $R/I$ is reduced; therefore, $\ass R/I$ is
the set of minimal prime ideals containing $I$.  These are monomial ideals,
and, hence, are in bijective correspondence with the set of \emph{minimal}
vertex covers of $G$. We will say that $G$ is \emph{unmixed} (respectively,
\emph{Cohen-Macaulay}) if $R/I$ is unmixed (respectively, Cohen-Macaulay).
Observe that if $G$ is unmixed, then all its minimal vertex covers have the
same size.

If $xy$ is an edge of $G$, then we say that $x$ and $y$ are
\emph{neighbours} of each other. An edge is \emph{incident} on its
vertices. We say that an edge $xy$ is \emph{isolated} if there are no other
edges incident on $x$ or on $y$. 
Let $G$ be a graph. A \emph{matching} in $G$ is a maximal (under inclusion)
set $\mathrm m$ of edges such that for all $x \in V$, at most one edge in
$\mathrm m$ is incident on $x$. Edges in a matching form a regular
sequence on $R$. We say that $G$ has \emph{perfect matching}, or, is
\emph{perfectly matched},
if there is a matching $\mathrm
m$ such that for all $x \in V$, exactly one edge in $\mathrm m$ is incident
on $x$.

\begin{lemma}
\label{thm:GPerfMatchIffCardsOfPartsIsHeight}
Let $G$ be a bipartite graph on the vertex set $V = V_1 \bigsqcup V_2$,
with edge ideal $I$. Then $G$ has a perfect matching if and only if $|V_1|
= |V_2| = \height I$. In particular, unmixed bipartite graphs have perfect
matching.
\end{lemma}

\begin{proof}
If $G$ has a perfect matching, then $|V_1| = |V_2|$. Moreover, by
K\"onig's theorem~\cite[Theorem~3.1.16]{WestGraphTheory96},
the maximum size of any matching equals the minimum size of any vertex
cover; hence $|V_1| = |V_2| = \height I$. Conversely, if $|V_1| = |V_2| =
\height I$, then, again by K\"onig's theorem, $G$ has a matching of $|V_1|
= |V_2|$ edges, \textit{i.e.}, it has a perfect matching.

If $G$ is unmixed, then every minimal vertex cover of $G$ has the same
size. Observe that both $V_1$ and $V_2$ are minimal vertex covers of $G$.
\end{proof}

\begin{discussion}
\label{disc:defn-of-digraph} 
Let $\mathfrak d$ be any directed graph on $[c]$, and denote the underlying
undirected graph of $\mathfrak d$ by $|\mathfrak d|$.  We will write $j
\succ i$ if there is a directed path from $i$ to $j$ in $\mathfrak d$. By
$j \succcurlyeq i$ (and, equivalently, $i \preccurlyeq j$) we mean that $j
\succ i$ or $j = i$. For $A \subseteq [c]$, we say that $j \succcurlyeq A$
if there exists $i \in A$ such that $j \succcurlyeq i$. We say that a set
$A \subseteq [c]$ is an \emph{antichain} if for all $i, j \in A$, there is
no directed path from $i$ to $j$ in $\mathfrak d$, and, by $\mathcal
A_\mathfrak d$, denote the set of antichains in $\mathfrak d$. We consider
$\varnothing$ as an antichain. A \emph{coclique} of $|\mathfrak d|$ is a
set $A \subseteq [c]$ such that for all $i \neq j \in A$, $i$ and $j$ are
not neighbours in $|\mathfrak d|$. Antichains in $\mathfrak d$ are
cocliques in $|\mathfrak d|$, but the converse is not, in general, true. We
say that $\mathfrak d$ is \emph{acyclic} if there are no directed cycles,
and \emph{transitively closed} if, for all $i, j, k \in [c]$, whenever $ij$
and $jk$ are (directed) edges in $\mathfrak d$, $ik$ is an edge. Observe
that $\mathfrak d$ is a poset under the order $\succcurlyeq$ if and only
if it is acyclic and transitively closed. If $\mathfrak d$ is a poset, we
say that, for $i, j \in [c]$, $j$ \emph{covers} $i$ if $j \succ i$ and
there does not exist $j'$ such that $j \succneqq j' \succneqq i$. Let $G$
be a bipartite graph on $V = V_1 \bigsqcup V_2$ with perfect matching. Let
$V_1 = \{x_1, \cdots, x_c\}$ and $V_2 = \{y_1, \cdots, y_c\}$.  After
relabelling the vertices, we will assume that $x_iy_i$ is an edge for all
$i \in [c]$. We associate $G$ with a directed graph $\mathfrak d_G$ on
$[c]$ defined as follows: for $i \neq j \in [c]$, $ij$ is an edge of
$\mathfrak d_G$ if and only if $x_iy_j$ is an edge of $G$.  (Here, by $ij$,
we mean the directed edge from $i$ to $j$.) Observe that $\mathfrak d_G$ is
simple, \textit{i.e.}, without loops and multiple edges. Let $\kappa(G)$
denote the largest size of any coclique in $|\mathfrak d_G|$.  
\hfill\qedsymbol
\end{discussion}

The significance of $\kappa(G)$ is that it gives a lower bound for $\reg
R/I$. Following Zheng~\cite{Zhereslnfacets04}, we say that two edges $vw$
and $v'w'$ of a graph $G$ are \emph{disconnected} if they are no more edges
between the four vertices $v,v',w,w'$. A set $\mathbf a$ of edges is
pairwise disconnected if and only if $(I \cap \Bbbk[V_\mathbf a])R$ is
generated by the regular sequence of edges in $\mathbf a$, where by
$V_\mathbf a$, we mean the set of vertices on which the edges in $\mathbf
a$ are incident.  The latter condition holds if and only if the subgraph of
$G$ induced on $V_\mathbf a$, denoted as $G|_{V_\mathbf a}$, is a
collection of $|\mathbf a|$ isolated edges. In particular, the edges in any
pairwise disconnected set form a regular sequence in $R$. Set $r(I) \defeq
\max \{ |\mathbf a| : \mathbf a \; \text{is a set of pairwise disconnected
edges in} \; G\}$.
\begin{lemma}
\label{thm:rIIsAtLeastKappaG}
Let $G$ be bipartite graph with perfect matching. Then, with notation as in
Discussion~\ref{disc:defn-of-digraph}, $r(I) \geq \kappa(G) \geq \max \{|A|
: A \in \mathcal A_{\mathfrak{d}_G}\}$.
\end{lemma}

\begin{proof}
If $A \subseteq [c]$ is a coclique of $|\mathfrak d_G|$, we easily see that
the edges $\{x_iy_i : i \in A\}$ are pairwise disconnected in $G$. The
assertion now follows from the observation, which we made in
Discussion~\ref{disc:defn-of-digraph}, that any antichain in $\mathfrak
d_G$ is a coclique of $|\mathfrak d_G|$.
\end{proof}

The assertion of Theorem~\ref{thm:unmBipRegIsrI} is that when $G$ is 
an unmixed bipartite graph, equality holds in the above lemma and that this
quantity equals $\reg R/I$. We will prove Theorem~\ref{thm:unmBipRegIsrI}
in the next section; now, we relate some properties of bipartite graphs
with their associated directed graphs.

\begin{lemma}
\label{thm:OnlyXorOnlyYonChains}
Let $G$ be bipartite graph with perfect matching, and adopt the notation of
Discussion~\ref{disc:defn-of-digraph}. Let $j \succcurlyeq i$. Then for all
$\mathfrak p \in \unm R/I$, if $y_i \in \mathfrak p$, then $y_j \in
\mathfrak p$.
\end{lemma}

\begin{proof}
Applying induction on the length of a directed path from $i$ to $j$, we may
assume, without loss of generality, that $ij$ is a directed edge of
$\mathfrak d_G$. Let $\mathfrak p \in \unm R/I$ and $k \in [c]$. Since
$x_ky_k \in I$, $x_k \in \mathfrak p$ or $y_k \in \mathfrak p$. Since
$\height \mathfrak p = c$, in fact, $x_k \in \mathfrak p$ if and only if
$y_k \not \in \mathfrak p$. Now since $y_i \in \mathfrak p$, $x_i \not \in
\mathfrak p$, so $(I:x_i) \subseteq \mathfrak p$. Note that since $x_iy_j$
is an edge of $G$, $y_j \in (I:x_i)$.
\end{proof}

\begin{thm}
\label{thm:VillarrealHerzogHibiBipGraphs}
Let $G$ be a bipartite graph on the vertex set $V = V_1 \bigsqcup V_2$. 
\begin{enumerate}
\label{enum:VillarrealHerzogHibiBipGraphs}
\item \label{thm:unmIsPerfMatchedAndTransCl} 
\cite[Theorem~1.1]{VillUnmBip07}
$G$ is unmixed if and only if $G$ has a perfect matching and $\mathfrak
d_G$ is transitively closed.
\item \label{thm:CMbipHeHinew} 
\cite[Lemma~3.3 and Theorem~3.4]{HeHiCMbip05}
$G$ is Cohen-Macaulay if and only if $G$ is perfectly matched and the
associated directed graph $\mathfrak d_G$ is acyclic and transitively
closed, \textit{i.e.}, it is a poset.
\hfill\qedsymbol
\end{enumerate}
\end{thm}

\begin{discussion}
\label{disc:partnIntoMaxlDirCyc}
Let $\mathfrak d$ be a directed graph. We say that a pair $i, j$ of
vertices $\mathfrak d$ are \emph{strongly connected} if there are directed
paths from $i$ to $j$ and from $j$ to $i$;
see~\cite[Definition~1.4.12]{WestGraphTheory96}.  A \emph{strong component}
of $\mathfrak d$ is an induced subgraph maximal under the property that
every pair of vertices in it is strongly connected.  Strong components of
$\mathfrak d$ form a partition of its vertex set.  Now let $G$ be a 
bipartite graph with perfect matching. Let $\mathcal Z_1, \ldots, \mathcal
Z_t$ be the vertex sets of the strong components of $\mathfrak d_G$. Define
a directed graph $\widehat{\mathfrak d}$ on $[t]$ by setting, for $a \neq b
\in [t]$, $ab$ to be a directed edge (from $a$ to $b$) if there exists a
directed path in $\mathfrak d_G$ from any (equivalently, all, since
$\mathfrak d_G|_{\mathcal Z_a}$ is strongly connected) of the vertices in
$\mathcal Z_a$ to any (equivalently, all, since $\mathfrak d_G|_{\mathcal
Z_b}$ is transitively closed) of the vertices in $\mathcal Z_b$. We observe
that $\widehat{\mathfrak d}$ has no directed cycles. Now assume further
that $G$ is unmixed. Then, since $\mathfrak d_G$ is transitively closed,
$\widehat{\mathfrak d}$ is transitively closed, \textit{i.e.}, it is a
poset under the order induced from $\mathfrak d_G$. We will use the same
notation for the induced order, \textit{i.e.}, say that $b \succ a$ if
there is a directed edge from $a$ to $b$. Define the \emph{acyclic
reduction} of $G$ to be the bipartite graph $\widehat G$ on new vertices
$\{u_1, \ldots, u_t\} \bigsqcup \{v_1, \ldots, v_t\}$, with edges $u_av_a$,
for all $1 \leq a \leq t$ and $u_av_b$, for all directed edges $ab$ of
$\widehat{\mathfrak d}$. Let $S = \Bbbk[u_1, \ldots, u_t, v_1, \ldots,
v_t]$, with standard grading. Let $\widehat I \subseteq S$ be the edge
ideal of $\widehat G$. Let $\zeta_i = |\mathcal Z_i|, 1 \leq i \leq t$. For
a multidegree $\tau = \prod_i u_i^{s_i} \prod v_i^{t_i}$, set
$\tau^\zeta = \prod_i u_i^{s_i\zeta_i} \prod v_i^{t_i\zeta_i}$.
\hfill\qedsymbol
\end{discussion}

\begin{lemma}
\label{thm:assIandAntiCh}
Let $G$ be an unmixed bipartite graph with edge ideal $I$. For an antichain
$A \neq \varnothing$ of $\widehat{\mathfrak d}$, let $\Omega_A = \{ j \in
\mathcal Z_b : b \succcurlyeq A\}$. Let $\Omega_\varnothing = \varnothing$.
Then $\ass R/I = \{(x_i : i \not \in \Omega_A) + (y_i : i \in \Omega_A) : A
\in \mathcal A_{\widehat{\mathfrak d}}\}$.
\end{lemma}

\begin{proof}
Let $\mathfrak p \in \ass R/I$. Let $\mathrm U \defeq \{b : y_j \in
\mathfrak p \ \text{for some}\ j \in \mathcal Z_b\}$. It follows from
Lemma~\ref{thm:OnlyXorOnlyYonChains} that $y_j \in \mathfrak p$ for all $j
\in \bigcup_{b \in \mathrm U} \mathcal Z_b$ and that if $b' \succ b$ for
some $b \in \mathrm U$, then $b' \in \mathrm U$. Now, the minimal elements
of $U$ form an antichain $A$ under $\succ$. Hence $\{j : y_j \in
\mathfrak p\} = \Omega_A$, showing $\ass R/I \subseteq \{(x_i : i \not \in
\Omega_A) + (y_i : i \in \Omega_A) : A \in \mathcal A_{\widehat{\mathfrak
d}}\}$.

Conversely, let $A \in \mathcal A_{\widehat{\mathfrak d}}$ and $\mathfrak p
\defeq (x_i : i \not \in \Omega_A) + (y_i : i \in \Omega_A)$. Since $\height
\mathfrak p = c = \height I$, it suffices to show that $I \subseteq
\mathfrak p$ in order to show that $\mathfrak p \in \ass R/I$. Clearly, for
all $1 \leq i \leq c$, $x_iy_i \in \mathfrak p$. Take $i \neq j$ such that
$x_iy_j \in I$. If $i \not \in \Omega_A$, then there is nothing to be
shown. If $i \in \Omega_A$, then there exist $a, b, b'$ such that $a \in
A$, $b \succ a$, $i \in \mathcal Z_b$ and $j \in \mathcal Z_{b'}$. Since
$ij$ is a
directed edge of $\mathfrak d_G$, $b' \succ b$ in $\widehat{\mathfrak d}$.
Hence $b' \succ a$, and $j \in \Omega_A$, giving $y_j \in \mathfrak p$. This
shows that $I \subseteq \mathfrak p$.
\end{proof}

\section{Regularity and Depth}
\label{sec:regDepthBipGr}

The content of Lemma~\ref{thm:assIandAntiCh} is that there are subsets $W
\subseteq V$ such that for all $\mathfrak p \in \ass R/I$, if $\mathfrak p
\cap W \neq \varnothing$ then $W \subseteq \mathfrak p$. Looking at
$I^\star$, we see that for all minimal generators $g$ of $I^\star$, if any
element of $W$ divides $g$, then all elements of $W$ divide $g$. Label the
minimal monomial generators of $I^\star$ as $g_1, \ldots, g_s, g_{s+1},
\ldots, g_m$ so that every element of $W$ divides $g_1, \ldots, g_s$ and
no element of $W$ divides $g_{s+1}, \ldots, g_m$. Fix $x \in W$. For $i=1,
\ldots, s$, set $h_i \defeq \frac{x^{|W|}}{\prod_{y \in W} y}g_i$ and $\bar
h_i \defeq \frac{x}{\prod_{y \in W} y}g_i$. Let $J = (h_1, \ldots, h_s,
g_{s+1}, \ldots, g_m)$ and $J' = (\bar h_1, \ldots, \bar h_s, g_{s+1},
\ldots, g_m)$. Let $\phi : R \rightarrow R$ be the ring homomorphism that
sends $x \mapsto x^{|W|}$ and $y \mapsto y$, for all $y \neq x \in V$. We
make two observations: first, that $I^\star$ is a polarization of $J$, and,
secondly, that $J = \phi(J')$. Hence the $\naturals$-graded
Betti numbers of $I^\star$ and $J$ are
identical~\cite[Exercise~3.15]{MiStCCA05}. Further, the following lemma
shows that $\beta_{l, \sigma}(R/J) \neq 0$ if and only if $x^{|W|}$ divides
$\sigma$ and $\beta_{l, \frac{\sigma}{x^{|W|-1}}}(R/J) \neq 0$.

\begin{lemma}
Let $B_1 = \Bbbk[x_1, \ldots, x_n]$ and $B_2= \Bbbk[y_1, \ldots, y_n]$. Let
$\xi_1, \ldots, \xi_n$ be positive integers. Set $\deg x_i = 1$ and $\deg
y_i = \xi_n$ for all $1 \leq i \leq n$. Define a ring homomorphism $\phi:
B_2 \rightarrow B_1$ by sending $y_i \mapsto x_i^{\xi_i}$.  Then for any
acyclic complex $\mathbb G_\bullet$ of finitely generated graded
$B_2$-modules (with degree-preserving maps), $\mathbb G_\bullet
\otimes_{B_2} B_1$ is an acyclic complex of finitely generated graded
$B_1$-modules (with degree-preserving maps).
\end{lemma}

\begin{proof}
Acyclicity of $\mathbb G_\bullet \otimes_{B_2} B_1$ follows from the fact
that $B_1$ is a free and hence flat $B_2$-algebra. The maps in 
$\mathbb G_\bullet \otimes_{B_2} B_1$ are degree-preserving since $\phi$
preserves degrees.
\end{proof}

\begin{propn}
\label{thm:bipPDimWeighedDepth}
Let $G$ be an unmixed bipartite graph, with edge ideal $I$ and acyclic
reduction $\widehat G$. Let $\widehat I \subseteq S$ be the edge ideal of
$\widehat G$. Then $\reg R/I = \projdim \left(\widehat I\right)^\star$ and
$\projdim R/I = \max \{|\tau^\zeta| - l : \beta_{l,\tau}
\left(\left(\widehat I\right)^\star\right) \neq 0 \}$.
\end{propn}

\begin{proof}
By Proposition~\ref{thm:TeraiPdimReg}, $\reg R/I = \projdim I^\star$ and
$\projdim R/I = \reg I^\star$. Hence it suffices to show that $\projdim
I^\star = \projdim \left(\widehat I\right)^\star$ and $\reg I^\star = \max
\{|\sigma^\zeta| - l : \beta_{l,\sigma} \left(\left(\widehat
I\right)^\star\right) \neq 0 \}$.  From Lemma~\ref{thm:assIandAntiCh}, with
the notation used there, it follows that 
\[
I^\star = \left(\prod_{i \not \in \Omega_A} x_i \cdot \prod_{i \in
\Omega_A} y_i : A \in \mathcal A_{\widehat{\mathfrak d}} \right)
= \left(\prod_{\substack{b \not \succcurlyeq A \\ i \in \mathcal Z_b}} x_i
\cdot 
\prod_{\substack{b \succcurlyeq A \\ i \in \mathcal Z_b}} y_i
: \varnothing \neq A \in \mathcal A_{\widehat{\mathfrak d}}
\right) + \left(\prod_{i=1}^c x_i\right).
\]
For each $a \in [t]$, fix $i_a \in \mathcal Z_a$. Now, as the $\mathcal
Z_a$ form a partition of $[c]$, we see that $I^\star$ is a polarization of
the ideal 
\begin{align*}
J = \left(\prod_{b \not \succcurlyeq A} x_{i_b}^{\zeta_b}
\cdot \prod_{b \succcurlyeq A} y_{i_b}^{\zeta_b}
: \varnothing \neq A \in \mathcal A_{\widehat{\mathfrak d}} \right)
+ \left(\prod_{b=1}^t x_{i_b}^{\zeta_b} \right) \\
\subseteq S \defeq \Bbbk[x_{i_1}, \ldots, x_{i_t}, y_{i_1}, \ldots,
y_{i_t}]
\end{align*}
Notice that $S' \simeq S$ (which, we recall, is the
polynomial ring on the vertex set of the acyclic reduction $\widehat G$)
under the map $\phi : x_{i_a} \mapsto u_a$ and $\phi:y_{i_a} \mapsto v_a$,
and that $\phi(\sqrt J) = \left(\widehat I\right)^\star$. Therefore
$\beta_{l, \sigma}(\sqrt J) = \beta_{l,\phi(\sigma)} \left(\left(\widehat
I\right)^\star\right)$. It now suffices to show that $\projdim J = 
\projdim \sqrt J$ and that $\reg J = \max \{|\tau^\zeta| - l :
\beta_{l,\tau}(\sqrt J) \neq 0 \}$. This, being the same argument as in the
opening paragraph of this section, follows from the preceding lemma.
\end{proof}

\begin{remark}
\label{rmk:regOfUnmGrAndAcycRedn}
Let $G$ be an unmixed graph with acyclic reduction $\widehat G$. If $I
\subseteq R$ and $\widehat I \subseteq S$ are the respective edge ideals,
then it follows from Proposition~\ref{thm:bipPDimWeighedDepth} that $\reg
R/I = \projdim \left(\widehat I\right)^\star = \reg S/\widehat I$.
\end{remark}

\begin{lemma}
\label{thm:maxSizeOfAntiChAcycRedn}
Let $G$ be an unmixed bipartite graph with acyclic reduction $\widehat G$.
Then $\max\{|A| : A \in \mathcal A_{\mathfrak{d}_G}\} = \max\{|A| : A \in
\mathcal A_{\mathfrak{d}_{\widehat G}}\}$. 
\end{lemma}

\begin{proof}
Let $A = \{i_1, \ldots, i_r\} \subseteq [c]$ be an antichain in $\mathfrak
d_G$. Choose $a_1, \ldots, a_r \in [t]$ such that $i_j \in \mathcal
Z_{a_j}$. Since $\mathfrak d_G$ is transitively closed, it follows that
$\{a_1, \ldots, a_r\}$ is an antichain in $\mathfrak d_{\widehat G}$.
Conversely, if $\{a_1, \ldots, a_r\}$ is an antichain in $\mathfrak
d_{\widehat G}$, then for any choice of $i_j \in \mathcal Z_{a_j}$, $\{i_1,
\ldots, i_r\}$ is an antichain in $\mathfrak d_G$.
\end{proof}

We now prove Theorem~\ref{thm:unmBipRegIsrI}. If $G$ is a tree --- trees are
bipartite --- then $\reg R/I$ is the maximum size of a pairwise
disconnected set of edges in $G$, without the assumption that $G$ is
unmixed~\cite[Theorem~2.18]{Zhereslnfacets04}. However, for bipartite
graphs $G$ that are not trees, we need to assume that $G$ is unmixed. For
example, if $G$ is the cycle on eight vertices, we can choose at most two
edges that are pairwise disconnected, while $\reg R/I = 3$.
{\def\thethm{\ref{thm:unmBipRegIsrI}}
\begin{thm}
\unmBipRegIsrIThm
\end{thm}
\addtocounter{thm}{-1}}

\begin{proof}
Since $\reg R/I \geq r(I)$ (see the paragraph on
page~\pageref{thm:unmBipRegIsrI} following the statement of 
Theorem~\ref{thm:unmBipRegIsrI}), 
the latter statement follows from the first statement along with
Lemma~\ref{thm:rIIsAtLeastKappaG}. In order to prove the first statement,
let $\widehat G$ be the acyclic reduction of $G$ on the vertex set $\{u_1,
\ldots, u_t\} \bigsqcup \{v_1, \ldots, v_t\}$. Recall that $\widehat G$ is
a Cohen-Macaulay bipartite graph. As in
Discussion~\ref{disc:partnIntoMaxlDirCyc}, let $S = \Bbbk[u_1, \ldots, u_t,
v_1, \ldots, v_t]$. Let $\widehat I \subseteq S$ to be the edge ideal of
$\widehat G$. Remark~\ref{rmk:regOfUnmGrAndAcycRedn} and
Lemma~\ref{thm:maxSizeOfAntiChAcycRedn} give that it suffices to prove the
theorem for Cohen-Macaulay bipartite graphs. If $G$ is Cohen-Macaulay,
then $\mathfrak d_G$ is a poset.
From~\cite[Corollary~2.2]{HeHiCMbip05}, taken along with
Proposition~\ref{thm:TeraiPdimReg}, we see that $\projdim R/I = \max\{|A| :
A \in \mathcal A_{\mathfrak{d}_G}\}$. (Note that $I^\star$ is the ideal
$H_{\mathfrak d_G}$, in the notation of~\cite{HeHiCMbip05}, with 
the $x_i$ and the $y_j$ interchanged.)
\end{proof}

\begin{remark}
\label{rmk:htIsRegIffCI}
Let $G$ be a Cohen Macaulay bipartite graph with edge ideal $I$, with
$\height I = c$. Then $\reg R/I \leq c$. If $\reg R/I = c$, then $R/I$ is a
complete intersection, or, equivalently, $G$ consists of $c$ isolated
edges. We see this as below: Let $\mathfrak d_G$ be the associated directed
graph on $[c]$. Since $\reg R/I$ is the maximum size of an antichain in
$\mathfrak d_G$, $\reg R/I \leq c$. If $\reg R/I = c$, we see
that $\mathfrak d_G$ has an antichain of $c$ elements, which implies that
for all $i \neq j \in [c]$, $i \not \succcurlyeq j$ or $j \not \succcurlyeq
i$, \textit{i.e.}, $x_iy_j$ is not an edge of $G$.
\end{remark}

We would now like to give a description of $\depth R/I$ for an unmixed
bipartite edge ideal $I$ in terms of the associated directed graph. First,
we determine the multidegrees with non-zero Betti numbers for its Alexander
dual. Let $G$ be a Cohen-Macaulay bipartite graph. For antichains $B
\subseteq A$ of $\mathfrak d_G$, $A \neq \varnothing$, set $\sigma_{A,B}
\defeq \prod_{i \not \succcurlyeq A}x_i \prod_{i \succcurlyeq A} y_i
\prod_{i \in B} x_i$.  Set $\sigma_{\varnothing, \varnothing} =
\prod_{i=1}^c x_i$. With this notation, we restate 
\cite[Theorem~2.1]{HeHiCMbip05} as follows:
\begin{thm}
\label{thm:multBettiNumForDual}
Let $G$ be a Cohen-Macaulay bipartite graph with edge ideal $I$. For all $l
\geq 0$, and multidegrees $\sigma$, if $\beta_{l,\sigma}(I^\star) \neq 0$,
then $\beta_{l,\sigma}(I^\star) = 1$ and $\sigma = \sigma_{A,B}$ for some
antichains $B \subseteq A$ of $\mathfrak d_G$ with $|B| = l$.
\end{thm}

(Although the multidegrees in which the Betti numbers are non-zero are not
explicitly given in the statement of \cite[Theorem~2.1]{HeHiCMbip05}, we
can determine then easily from the description of the differentials given
there, prior to stating the theorem. Note, again, that the roles of the
$x_i$ and the $y_j$ are the opposite of what we follow.)

\begin{cor}
\label{thm:depthBdInTheZeta}
Let $G$ be an unmixed bipartite graph with edge ideal $I$. Let $c = \height
I$. Let $t, \zeta_1, \ldots, \zeta_t, \widehat{\mathfrak d}$ be as in
Discussion~\ref{disc:partnIntoMaxlDirCyc}. Then 
\[
\depth R/I = c - \max \left\{\sum_{i \in B} \zeta_i - |B| :
B \ \text{is an antichain of} \ \widehat{\mathfrak d} \right\}.
\]
\end{cor}

\begin{proof}
Let $\widehat G, S, \widehat I$ be as in
Discussion~\ref{disc:partnIntoMaxlDirCyc}. From
Theorem~\ref{thm:multBettiNumForDual}, we know that if
$\beta_{l,\sigma}((\widehat I)^\star) \neq 0$ for some multidegree $\sigma
\subseteq \{u_1, \ldots, u_t, v_1, \ldots, v_t\}$, then $\sigma =
\sigma_{A,B}$ for some antichains $B \subseteq A$ of $\widehat{\mathfrak
d}$, with $|B| = l$. Now, in $S$, $\deg \sigma_{A,B} = \sum_{i \succcurlyeq
A} \zeta_i + \sum_{i \not \succcurlyeq A} \zeta_i + \sum_{i \in B} \zeta_i
= c + \sum_{i \in B} \zeta_i$. Hence 
\[
\reg (\widehat I)^\star = c + \max \left\{\sum_{i \in B} \zeta_i - |B| : B
\ \text{is an antichain of} \ \widehat{\mathfrak d} \right\}.
\]
Note that $\depth R = \dim R = 2c$. Now apply 
Proposition~\ref{thm:bipPDimWeighedDepth}, followed by the
Auslander-Buchsbaum formula, to obtain the conclusion.
\end{proof}

The above proof also shows that if $G$ is a bipartite graph such that $R/I$
satisfies Serre's condition $\serreS2$ (defined, \textit{e.g.},
in~\cite[Section~2.1]{BrHe:CM}) then $G$ is Cohen-Macaulay. For, if $R/I$
satisfies $\serreS2$, then it is unmixed and $I^\star$ is linearly
presented, \textit{i.e.}, the non-zero entries in any matrix giving a
presentation of $I^\star$ has linear entries. This is a special case
of~\cite[Corollary~3.7]{YanaSRringsDuality00}. It follows, with the
notation of the proof, that for all antichains $A \neq \varnothing$ of
$\widehat{\mathfrak d}$, and for all $a \in A$, $\deg \sigma_{A, \{a\}} = c
+ \zeta_a = c+1$, giving that every strong component of $\mathfrak d_G$ has
exactly one element. In other words, $G$ is Cohen-Macaulay. We can now
prove Theorem~\ref{thm:unmBipDepthWeakVersion}.
{\def\thethm{\ref{thm:unmBipDepthWeakVersion}}
\begin{thm}
\unmBipDepthWeakVersion
\end{thm}
\addtocounter{thm}{-1}}

\begin{proof}
To show that $\depth R/I \geq t$, it suffices to show that, for all
antichains $B$ of $\widehat{\mathfrak d}$, $t + \sum_{i \in B} \zeta_i -
|B| \leq c$. Since $c = \sum_{i=1}^t \zeta_i$, it suffices to show that $t
- |B| \leq \sum_{i \not \in B} \zeta_i$, which is true since $\zeta_i \geq
  1$ for all $i$.
\end{proof}

\begin{remark}
\label{rmk:unmBipWithDepthBdEq}
The above bound is sharp. Given positive integers $t
\leq c$, and a poset $\widehat{\mathfrak d}$ on $t$ vertices, we can find
an unmixed bipartite graph $G$ on the vertex set $V = V_1 \bigsqcup V_2$
with edge ideal $I$ such that $|V_1| = |V_2| = c$ and $\depth \Bbbk[V]/I =
t$. Choose any antichain
$B$ in $\widehat{\mathfrak d}$ and set $\zeta_i = 1$ for all $i \not \in
B$. Choose $\zeta_i \geq 1, i \in B$ such that $\sum_{i \in B}  \zeta_i =
c-t+|B|$. Now construct a directed graph $\mathfrak d$ on $c$ vertices by
replacing the vertex $i$ of $\widehat{\mathfrak d}$ by directed cycle of
$\zeta_i$ vertices and then taking its transitive closure. Label the
vertices of $\mathfrak d$ with $[c]$. Let $G$ be a
bipartite graph on $V = \{x_1, \ldots, x_c\} \bigsqcup \{y_1, \ldots,
y_c\}$ such that $x_iy_i$ is an edge for all $i \in [c]$ and $x_iy_j$ is an
edge whenever $ij$ is a directed edge of $\mathfrak d$. Then $G$ is an
unmixed graph. We know from the corollary that $t \leq \depth R/I \leq c -
\sum_{i \in B} \zeta_i - |B| = t$.
\end{remark}

\section{Arithmetic Rank}
\label{sec:arithRk}

The two statements of Theorem~\ref{thm:compatLinznImpliesSTCI} will be
proved separately in Proposition~\ref{thm:araBreveIPlusDiff} and in
Proposition~\ref{thm:compatLinznImpliesSTCICMCase}.

\begin{discussion}
\label{disc:unmBipBreveGraph}
Let $G$ be an unmixed bipartite graph on $\{x_1, \ldots, x_c\} \bigsqcup
\{y_1, \ldots, y_c\}$. Adopt the notation of
Discussion~\ref{disc:partnIntoMaxlDirCyc}. Choose an acyclic
transitively closed subgraph of $\mathfrak d_G$ which is maximal under
inclusion of edge sets; call it $\breve{\mathfrak d}$. It is a poset, with
the order induced from $\mathfrak d_G$. We will denote this order by
$\vartriangleright$ to avoid confusion with $\succ$.  (Recall that $\succ$
does not define a partial order if $G$ is not Cohen-Macaulay.) Let
$\breve{G}$ be the Cohen-Macaulay bipartite graph on $\{x_1, \ldots, x_c\}
\bigsqcup \{y_1, \ldots, y_c\}$ corresponding to $\breve{\mathfrak d}$;
denote its edge ideal by $\breve{I}$. 
\hfill\qedsymbol
\end{discussion}

\begin{propn}
\label{thm:araBreveIPlusDiff}
With notation as above, $\arank I \leq \arank \breve{I} + \projdim R/I -
\height I$.
\end{propn}

\begin{proof}
On the set $\{x_jy_i : j \vartriangleright i, j \neq i\ \text{and} \ x_jy_i
\ \text{is an edge of} \ G\}$, define a partial
order: $x_jy_i > x_{j'}y_{i'}$ whenever $j \vartriangleright j', j \neq j',
i \vartriangleright i', i \neq i'$. Call this poset $P$. (These are the
edges of $G$ that do not belong to $\breve{G}$. If $x_jy_i$ is such an
edge, then $i$ and $j$ belong to the same strong component of $\mathfrak
d_G$.) We now claim that every
antichain in $P$ has at most $\max \left\{\sum_{a \in B} \zeta_a - |B| : B
\ \text{is an antichain of} \ \widehat{\mathfrak d} \right\}$ elements;
this quantity, as we note from Corollary~\ref{thm:depthBdInTheZeta}, equals
$\xi \defeq \projdim R/I - \height I$. Let $\{x_{j_k}y_{i_k} : 1 \leq k
\leq l\}$ with $j_k \vartriangleright i_k, 1 \leq k \leq l$ be an antichain
in $P$. First, there exist $a_1, \ldots, a_l$ such that $i_k,
j_k \in \mathcal Z_{a_k}$; this arises from the fact that $j_k
\vartriangleright i_k$. If $a_{k_2} \succneqq a_{k_1}$, then for, $i, j \in
\mathcal Z_{a_{k_1}}$ and $i', j' \in \mathcal Z_{a_{k_2}}$, $x_{j'}y_{i'}
> x_jy_i$, so if $a_{k_2} \neq a_{k_1}$, then they are incomparable.
Therefore, to prove the claim, it suffices to show that if $a_1 = \ldots =
a_l = a$, say, then $l \leq \zeta_a-1$. This follows easily, for, in this
case, any antichain in $P$ can contain at most one edge for each value of
$j-i$, and $1 \leq j-i \leq \zeta_a-1$. Moreover, let $B$ be an antichain
of $\widehat{\mathfrak d}$ for which the maximum is attained.  For all $a
\in B$, set $j_a$ to be the maximal element of $\mathcal Z_a$ under
$\vartriangleright$. Then $\{x_{j_a}y_i : i \in \mathcal Z_a, a \in B\}$ is
an antichain of $P$ with $\xi$ elements. Using Dilworth's
theorem~\cite[p.~413]{WestGraphTheory96}, we cover $P$ with $\xi$
chains, $\mathcal C_1, \ldots, \mathcal C_\xi$. For $1 \leq k \leq \xi$,
set $h_k \defeq \sum_{x_jy_i \in \mathcal C_k} x_jy_i$.

Our final claim is that $\sqrt{\breve{I} + (h_1, \ldots, h_\xi)} =
I$. The $h_l$ belong to $I$ and $\breve I \subseteq I$, so it suffices to
show that $I \subseteq \mathfrak p$ for every $\mathfrak p \in \Spec R$
such that $\breve{I} + (h_1, \ldots, h_\xi) \subseteq \mathfrak p$. Let
$\mathfrak p$ be such, and, by way of contradiction, assume that $x_jy_i
\in I \setminus \mathfrak p$; since $\breve{I} \subseteq \mathfrak p$, $j
\vartriangleright i$. First, we may also assume that for all $i' \neq i, i
\vartriangleright i'$, if $x_jy_{i'} \in I$, then $y_{i'} \in \mathfrak p$,
and similarly, that
for all $j' \neq j, j \vartriangleright j'$, if $x_{j'}y_i \in I$, then
$x_{j'} \in \mathfrak p$. Secondly, $i$ and $j$ belong to the same strong
component of $\mathfrak d_G$; let $a$ be such that $i, j \in \mathcal Z_a$.
Let $\mathcal C_l$ be chain of $P$ containing $x_jy_i$. For all $b
\succneqq a$ and $j' \in \mathcal Z_b$, $x_jy_{j'} \in \breve{I} \subseteq
\mathfrak p$, so $y_{j'} \subseteq \mathfrak p$.  Similarly, for all $b
\precneqq a$ and $i' \in \mathcal Z_b$, $x_{i'}y_i \in \breve{I} \subseteq
\mathfrak p$, so $x_{i'} \subseteq \mathfrak p$. We can thus conclude that
if $x_{j'}y_{i'} \in \mathcal C_l$ and $(i,j) \neq (i',j')$, then
$x_{j'}y_{i'} \in \mathfrak p$. Therefore $x_jy_i \in \mathfrak p$,
contradicting the choice of $x_jy_i$.
\end{proof}

On $\naturals^2$, we define a poset by setting $(a,b) \geq (c,d)$ if $a
\geq c$ and $b \geq d$. Let $(P, \geq)$, be a finite poset on a vertex set
$W_1$. We say that $P$ can be \emph{embedded} in $\naturals^2$ if there
exists a map $\phi : W \longrightarrow \naturals^2$ such that all $i, j \in
W$, $j \geq i$ if and only if $\phi(j) \geq \phi(i)$; such a map $\phi$
will be called an \emph{embedding} of $P$ in $\naturals^2$. We will denote
the projection of $\naturals^2$ along the first co-ordinate by $\pi$.

\begin{defn}
\label{disc:rowAndColOrderingPosetsDefn}
Let $(P, \succcurlyeq)$ be a finite poset on a finite vertex set $W$, with
an embedding $\phi$ in $\naturals^2$. Then there is a unique $i_0 \in
W$ such that $i_0$ is minimal in $P$ and $(\pi \circ \phi)(i_0)$ is
minimum. Similarly, let $j_0$ be the unique maximal element such that 
$(\pi \circ \phi)(j_0)$ is minimum. Let $P_1$ and $P_2$ be the
restrictions of $P$ respectively to $W \setminus \{i_0\}$ and $W \setminus
\{j_0\}$.  The \emph{column linearization} of $P$ induced by $\phi$
is the map $\gamma : W \longrightarrow [|W|]$ defined recursively as follows:
\[
\gamma(i) = 
\begin{cases}
1, & i = i_0 \\
1+ \gamma_1(i), & i \neq i_0 \\
\end{cases}
\]
where $\gamma_1$ is a column linearization of $P_1$ induced by $\phi$.
A \emph{row linearization} of $P$ induced by $\phi$
is the map $\rho : W \longrightarrow [|W|]$ defined recursively as follows: \[
\rho(j) = 
\begin{cases}
1, & j = j_0 \\
1+ \rho_1(j), & j \neq j_0 \\
\end{cases}
\]
where $\rho_1$ is a row linearization of $P_2$ induced by $\phi$. We will
say that $(\gamma, \rho)$ is the pair of linearizations induced by $\phi$.
\hfill\qedsymbol
\end{defn}

\begin{propn}
\label{thm:RowColLinznsInduceEmbedding}
Let $P$, $\phi$, $\gamma$ and $\rho$ be as in
Definition~\ref{disc:rowAndColOrderingPosetsDefn}. For $i, j \in P$, if $j
\succcurlyeq i, j \neq i$, then $\gamma(j) > \gamma(i)$ and $\rho(j) <
\rho(i)$. If $i$ and
$j$ are incomparable, then $\gamma(j) > \gamma(i)$ if and only if $\rho(j)
> \rho(i)$.
\end{propn}

\begin{proof}
If $j \succcurlyeq i$, then $\phi(j) \geq \phi(i)$. In the recursive
definition of $\gamma$, $i$ would appear as the unique minimal vertex with
the smallest value of $(\pi \circ \phi)$ before $j$ would, so $\gamma(i) <
\gamma(j)$. On the other hand, while computing $\rho$ recursively, $j$
would appear as the unique maximal vertex with the smallest value of $(\pi
\circ \phi)$ before $i$ would, so $\rho(j) < \rho(i)$. On the other hand,
if $i$ and $j$ are incomparable, then we may assume without loss of
generality that $(\pi \circ \phi)(i) < (\pi \circ \phi)(j)$. Hence, while
computing $\gamma$ and $\rho$ recursively, $i$ will be chosen before $j$,
giving $\gamma(i) < \gamma(j)$ and $\rho(i) < \rho(j)$.
\end{proof}

\begin{discussion}
\label{disc:defnQuadFormsSTCI}
Let $P$ be a poset on a finite set $W$, with an embedding $\phi$ in
$\naturals^2$. Let $(\gamma, \rho)$ be the pair of linearizations of
$P$ induced by $\phi$. Let $E = \{(\gamma(i), \rho(j)) : j \succcurlyeq
i \in W\} \subseteq \reals^2$. We think of $E$ as
a subset of $[|W|] \times [|W|]$ in the first quadrant of the Cartesian
plane. Let $i, j$ be such that $(\gamma(i), \rho(j)) \in E$ is not the
lowest vertex in its column, \textit{i.e.}, there exists $l$ such that
$(\gamma(i), \rho(l))$ lies below $(\gamma(i), \rho(j))$. Then $j
\succcurlyeq i$, $l \succcurlyeq i$ and, from
Proposition~\ref{thm:RowColLinznsInduceEmbedding}, $l \neq i$. Therefore,
again from Proposition~\ref{thm:RowColLinznsInduceEmbedding}, $\gamma(l) >
\gamma(i)$ and $(\gamma(i), \rho(l))$ is not the right-most vertex in its
row.  Let $k$ be such that $(\gamma(k), \rho(l))$ lies immediately to the
right of $(\gamma(i), \rho(l))$ in its row. Draw an edge between
$(\gamma(i), \rho(j))$ and $(\gamma(k), \rho(l))$. Repeating this for all
$j \succcurlyeq i$ such that $(\gamma(i), \rho(j))$ is not the lowest
vertex in its column, we obtain a graph $\Gamma$ on $E$. Rows and columns
of $\Gamma$ will be indexed starting from the bottom left corner.
\hfill\qedsymbol
\end{discussion}

\begin{lemma}
\label{thm:compatLinznGammaCpnts}
With notation as in Discussion~\ref{disc:defnQuadFormsSTCI},
$\Gamma$ has exactly $|W|$ connected components.
\end{lemma}

\begin{proof}
Suppose that $C$ is a connected component of $\Gamma$ and that $(\gamma(i),
\rho(j))$ is the top left vertex of $C$. We claim that it is the left-most
vertex in its row. For, if not, then there exists $k$ such that
$(\gamma(k), \rho(j))$ lies immediately to the left of $(\gamma(i),
\rho(j))$. From Proposition~\ref{thm:RowColLinznsInduceEmbedding}, $k \neq
j$.  We note, again from Proposition~\ref{thm:RowColLinznsInduceEmbedding},
that $(\gamma(k), \rho(j))$ is not the top-most vertex in its column,
contradicting the hypothesis that that $(\gamma(i), \rho(j))$ is the top
left vertex of $C$.  Now, there are exactly $|W|$ rows in $\Gamma$.
\end{proof}

\begin{lemma}
\label{thm:FirstTwoColsUnderCompatLinzn}
Let $G$ be a Cohen-Macaulay bipartite graph such that $\phi$ is an
embedding of $\mathfrak d_G$ in $\naturals^2$. Let $(\gamma, \rho)$ be the
pair of linearizations induced by $\phi$. Then the vertices in the first
column of $\Gamma$ belong to a contiguous set of rows, starting with row
$1$.
%If $\gamma^{-1}(2) \not \succcurlyeq \gamma^{-1}(1)$, then
%$\rho(\gamma^{-1}(2)) > \rho(\gamma^{-1}(1))$, and the vertices in the
%second column belong to a contiguous set of rows, starting with row $1$,
%except in the row $\rho(\gamma^{-1}(1))$. If $\gamma^{-1}(2) \succcurlyeq
%\gamma^{-1}(1)$, then $\rho(\gamma^{-1}(2)) < \rho(\gamma^{-1}(1))$, and
%the vertices in the second column belong to a contiguous set of rows,
%starting with row $1$.
\end{lemma}

\begin{proof}
We may assume that the labelling of $\mathfrak d_G$ is such
that $\gamma^{-1}(1)  = 1$ and $\gamma^{-1}(2) = 2$. We need to show that
$\rho(i) > \rho(1)$ if $i \not \succcurlyeq 1$.
Proposition~\ref{thm:RowColLinznsInduceEmbedding} gives that $1$ is minimal in
$\mathfrak d_G$. Let $i \not \succcurlyeq 1$. Then $i$ and $1$ are
incomparable. Since $\gamma(1) = 1 \leq \gamma(i)$, we see,
again from Proposition~\ref{thm:RowColLinznsInduceEmbedding},
that $\rho(i) > \rho(1)$.
\end{proof}

\begin{remark}
\label{rmk:compatLinznRestrictionSTCI}
Let $P$ be a poset on a finite vertex set $W$ with an embedding $\phi$ in
$\naturals^2$. Let $(\gamma, \rho)$ be the pair of linearizations of $P$ induced
by $\phi$. Let $W' = W \setminus \{\gamma^{-1}(1)\}$ and let $P'$ be the
restriction of $P$ to $W'$. Then $\phi|_{W'}$ is an embedding of $P'$ in
$\naturals^2$. For $i \in W'$, set $\gamma'(i) = \gamma(i) - 1$, and 
\[
\rho'(i) = 
\begin{cases}
\rho(i), & i \succcurlyeq \gamma^{-1}(1) \\
\rho(i)-1, & \text{otherwise}.
\end{cases}
\]
Then $(\gamma', \rho')$ is the pair of linearizations induced by
$\phi|_{W'}$. Let $\Gamma'$ be the graph constructed from $P'$ as described
in Discussion~\ref{disc:defnQuadFormsSTCI} using $\gamma'$ and $\rho'$.
Then $\Gamma'$ is obtained by deleting the vertices in the first column of
$\Gamma$. We see this as follows. For all $i, j \in W'$, $\rho(i) <
\rho(j)$ if and only if $\rho'(i) < \rho'(j)$; similarly, $\gamma(i) <
\gamma(j)$ if and only if $\gamma'(i) < \gamma'(j)$. Further, there is only
one vertex in row $\rho(\gamma^{-1}(1))$ in $\Gamma$, and this is in the
first column.
\end{remark}

\begin{remark}
\label{rmk:compatLinznRestrictionSTCIColon}
Let $P$ be a poset on a finite vertex set $W$ with an embedding $\phi$ in
$\naturals^2$. Let $(\gamma, \rho)$ be the pair of linearizations induced
by $\phi$. Let $W' = W \setminus \gamma^{-1}(1)$ and let
$P'$ be the restriction of $P$ to $W'$. Then $\phi|_{W'}$ is an embedding
of $P$ in $\naturals^2$. Let $\tilde \gamma$ be the
order-preserving map from $\image \gamma|_{W'}$ to $[|W'|]$. 
Let $\gamma' \defeq \tilde \gamma \circ \gamma|_{W'}$.
For $j \in W'$, set $\rho'(j) = \rho(j) - \rho(1)$.
Then $(\gamma', \rho')$ is the pair of linearizations of $P'$
induced by $\phi|_{W'}$. Let $\Gamma'$ be the graph constructed from $P'$
as described in Discussion~\ref{disc:defnQuadFormsSTCI} using $\tilde
\gamma \circ \gamma|_{W'}$ and $\tilde \rho \circ \rho|_{W'}$. We claim
that $\Gamma'$ is the graph obtained from $\Gamma$ by deleting the vertices
that lie in rows $\rho(j)$ for some $j \succcurlyeq \gamma^{-1}(1)$. For,
first observe that for all $i, j \in W'$, $\rho(i) < \rho(j)$ if and only
if $\rho'(i) < \rho'(j)$; similarly, $\gamma(i) < \gamma(j)$ if and only if
$\gamma'(i) < \gamma'(j)$. Moreover, for all $j \succcurlyeq
\gamma^{-1}(1)$, the vertices in the column $\gamma(j)$ belong to rows
between $1$ and $\rho(j)$ (possibly, not all of them). Therefore, after the
vertices in the rows between $1$ and $\rho(1)$ have been deleted, the
remaining vertices belong to columns $\gamma(j)$ for $j \not \succcurlyeq
1$. Hence $(\gamma'(i), \rho'(j))$ and $(\gamma'(k), \rho'(l))$ belong
to the same connected component of $\Gamma'$ if and only if $(\gamma(i),
\rho(j))$ and $(\gamma(k), \rho(l))$ belong to the same connected component
of $\Gamma$.
\end{remark}

\begin{example}
\label{example:posetExampleOne}
We wish to illustrate these constructions with an example of a
Cohen-Macaulay bipartite graph. Let $G$ be the Cohen-Macaulay bipartite
graph on the vertex set $\{x_1, y_1, \ldots, x_7, y_7\}$ such that the
poset $\mathfrak d_G$ has the cover relations (\textit{i.e.}, chains that
cannot be further refined) $3 \succ 1$, $3 \succ 2$ $4 \succ 1$, $4 \succ
2$, $5 \succ 2$, $6 \succ 3$, $6 \succ 4$, $7 \succ 4$ and $7 \succ 5$.
Table~\ref{tab:exampleEmb} gives the embedding $\phi$, the functions
$\gamma$ and $\rho$ and the graph $\Gamma$.
We take the sum of the monomials corresponding to the vertices in a
connected component of $\Gamma$:
\begin{align*}
g_1 & = x_1y_6, \qquad g_2 = x_2y_6 + x_1y_3, &
g_3 & = x_3y_6 + x_2y_3 + x_1y_7, \\
g_4 & = x_4y_6 + x_3y_3 + x_2y_7 + x_1y_4, &
g_5 & = x_6y_6 + x_4y_7 + x_2y_4 + x_1y_1, \\
g_6 & = x_5y_7 + x_4y_4 + x_2y_5, &
g_7 & = x_7y_7 + x_5y_5 + x_2y_2.\\
\end{align*}
Let $J = (g_1, \ldots, g_7)$. In the proof of 
Proposition~\ref{thm:compatLinznImpliesSTCICMCase} we will see that $I =
\sqrt{J}$.
\hfill\qedsymbol
\end{example}
\begin{table}
\caption{Example~\protect{\ref{example:posetExampleOne}}}
\label{tab:exampleEmb}
\begin{minipage}[c]{0.495 \textwidth}
\centering
\begin{tabular}{|c|ccc|}
\hline
$i$ & $\phi(i)$ & $\gamma(i)$ & $\rho(i)$ \cr
\hline
$1$ & $(0,2)$ & $1$ & $5$ \cr
$2$ & $(1,0)$ & $2$ & $7$ \cr
$3$ & $(2,5)$ & $3$ & $2$ \cr
$4$ & $(3,3)$ & $4$ & $4$ \cr
$5$ & $(5,1)$ & $6$ & $6$ \cr
$6$ & $(4,6)$ & $5$ & $1$ \cr
$7$ & $(6,4)$ & $7$ & $3$ \cr
\hline
\end{tabular}
\end{minipage}
\hfill
\begin{minipage}[c]{0.495 \textwidth}
\centering
\setlength{\unitlength}{2pt}
\begin{picture}(74, 74)
\put(1,9){$y_6$}
\put(1,19){$y_3$}
\put(1,29){$y_7$}
\put(1,39){$y_4$}
\put(1,49){$y_1$}
\put(1,59){$y_5$}
\put(1,69){$y_2$}
\put(9,1){$x_1$}
\put(19,1){$x_2$}
\put(29,1){$x_3$}
\put(39,1){$x_4$}
\put(49,1){$x_6$}
\put(59,1){$x_5$}
\put(69,1){$x_7$}
\put(10,10){\circle*{1}}
\put(10,20){\circle*{1}\line(1,-1){10}}
\put(10,30){\circle*{1}\line(1,-1){10}}
\put(10,40){\circle*{1}\line(1,-1){10}}
\put(10,50){\circle*{1}\line(1,-1){10}}
\put(20,10){\circle*{1}}
\put(20,20){\circle*{1}\line(1,-1){10}}
\put(20,30){\circle*{1}\line(1,-1){10}}
\put(20,40){\circle*{1}\line(2,-1){20}}
\put(20,60){\circle*{1}\line(1,-1){20}}
\put(20,70){\circle*{1}\line(4,-1){40}}
\put(30,10){\circle*{1}}
\put(30,20){\circle*{1}\line(1,-1){10}}
\put(40,10){\circle*{1}}
\put(40,30){\circle*{1}\line(1,-2){10}}
\put(40,40){\circle*{1}\line(2,-1){20}}
\put(50,10){\circle*{1}}
\put(60,30){\circle*{1}}
\put(60,60){\circle*{1}\line(1,-3){10}}
\put(70,30){\circle*{1}}
\end{picture}
\end{minipage}
\vskip 1em
\end{table}

Before we prove the second assertion of
Theorem~\ref{thm:compatLinznImpliesSTCI}, we observe that the directed
graph associated to $\breve G$ (which we denoted by $\breve{\mathfrak d}$
in Discussion~\ref{disc:unmBipBreveGraph}) has an embedding in
$\naturals^2$ if and only if the acyclic reduction $\widehat{\mathfrak
d}$ of $\mathfrak d_G$ has an embedding in $\naturals^2$. The proof of this
is easy, and is omitted.

\begin{propn}
\label{thm:compatLinznImpliesSTCICMCase}
Let $G$ be an unmixed bipartite graph. If a maximal transitively closed and
acyclic subgraph of $\mathfrak d_G$ can be embedded in $\naturals^2$, then
$\arank I = \projdim R/I$.
\end{propn}
\begin{proof}
Let $\breve{\mathfrak d}$ be a maximal acyclic subgraph of $\mathfrak d_G$
with the property that $\breve{\mathfrak d}$ can be embedded in
$\naturals^2$. Construct $\breve{G}$ as in
Discussion~\ref{disc:unmBipBreveGraph}. Let $\breve{I}$ be its edge ideal.
Observe that $\breve{G}$ is Cohen-Macaulay and $\height \breve I = \height
I = c$. Suppose that the conclusion of the proposition holds for
Cohen-Macaulay graphs. Then $\arank \breve I = \projdim R/\breve I =
\height I$. Using Proposition~\ref{thm:araBreveIPlusDiff} and the fact that
$\arank I \geq \projdim R/I$ (\cite[Proposition ~3]{LyubArithRk88}), we
conclude that $\arank I = \projdim R/I$. Hence it suffices to prove the
assertion in the Cohen-Macaulay case. Assume, therefore, that $G$ is
Cohen-Macaulay.

Denote the embedding of $\mathfrak d_G$ by $\phi$, and let $(\gamma, \rho)$
be pair of linearizations induced by $\phi$. Let $\Gamma$ be the graph
constructed as in Discussion~\ref{disc:defnQuadFormsSTCI}. We prove the
theorem by induction on $c$. Since the conclusion is evident when
$c=1$, we assume that $c>1$ and that it holds for all Cohen-Macaulay
bipartite graphs on fewer than $2c$ vertices. For $t = 1, \ldots, c$, let
$C_t$ be the connected component of $\Gamma$ containing the left most
vertex in row $t$. We saw in the proof of
Lemma~\ref{thm:compatLinznGammaCpnts}
that these are exactly 
the connected components of $\Gamma$. Set 
\[
g_t = \sum_{(\gamma(i), \rho(j)) \in C_t} x_iy_j\qquad  1 \leq t \leq c.
\]
Set $J = (g_1, \ldots, g_c)$. We will show that $I = \sqrt J$, or,
equivalently, that for all $\mathfrak p \in \Spec R$, $I \subseteq
\mathfrak p$ if and only if $J \subseteq \mathfrak p$. (This gives that
$\arank I \leq c = \projdim R/I = \height I$, but we have already noted
that $\arank I \geq = \projdim R/I$.) Further, without
loss of generality, we may assume that $\gamma^{-1}(1) = 1$. Then $1$ is a
minimal element of $\mathfrak d_G$. Let $W_1 \defeq \{2, \ldots, c\}$ and
$W_2 \defeq \{i \not \succcurlyeq 1\} \subseteq [c]$. Let $\mathfrak d_1$
and $\mathfrak d_2$ respectively be the restrictions of $\mathfrak d_G$ to
$W_1$ and $W_2$. 

Let $G_1$ be the deletion of $x_1$ and $y_1$ in $G$, whose edge ideal (in
$R = \Bbbk[V]$) is $((I, x_1) \cap \Bbbk[x_2, y_2, \ldots, x_c, y_c])R$.
Note that $\mathfrak d_1$ is the associated directed graph of $G_1$. Let
$\Gamma_1$ denote the deletion of the vertices that lie in the first column
of $\Gamma$. Write $J_1 = ((J, x_1) \cap \Bbbk[x_2,
y_2, \ldots, x_c, y_c])R$. We see from
Remark~\ref{rmk:compatLinznRestrictionSTCI} that 
that $J_1$ is defined from $\Gamma_1$
precisely the same way that $J$ is defined from $\Gamma$. Along with
the induction hypothesis, this gives that $((I, x_1) \cap \Bbbk[x_2, y_2,
\ldots, x_c, y_c])R = \sqrt{J_1}$. Note that $(J_1, x_1) = (J,x_1)$, so we
obtain that $(I, x_1) = \sqrt{(J,x_1)}$. We thus see that for all
$\mathfrak p \in \Spec R$ such that $x_1 \in \mathfrak p$, $I \subseteq
\mathfrak p$ if and only if $J \subseteq \mathfrak p$.

Let $G_2$ be the deletion of $x_1$ and all its neighbours in $G$; its edge
ideal is $((I:x_1) \cap \Bbbk[x_i, y_i: i \in W_2])R$. The associated
directed graph of $G_2$ is $\mathfrak d_2$. Let $\Gamma_2$ denote the
deletion of the vertices that lie in columns $\gamma(i)$ or in rows
$\rho(i)$ of $\Gamma$ whenever $i \succcurlyeq 1$. Let
\[
J_2 = ((J + (y_i : i \succcurlyeq 1)) \cap
\Bbbk[x_i, y_i : i \not \succcurlyeq 1])R.
\]
From Remark~\ref{rmk:compatLinznRestrictionSTCI}, we note that 
$J_2$ is defined from $\Gamma_2$ precisely the same way that $J$ is
defined from $\Gamma$. This, along with the induction
hypothesis, implies that $((I:x_1) \cap \Bbbk[x_i, y_i: i \in W_2])R = 
\sqrt{J_2}$. Now, $J_2 + (y_i : i \succcurlyeq 1) = J + (y_i : i
\succcurlyeq 1) = (J:x_1)$, so $(I:x_1) = \sqrt{(J:x_1)}$.  We thus see
that for all $\mathfrak p \in \Spec R$ such that $x_1 \not \in \mathfrak
p$, $I \subseteq \mathfrak p$ if and only if $J \subseteq \mathfrak p$.
Together with the previous paragraph, we conclude that $\sqrt J = I$.
\end{proof}

\section*{Acknowledgments}
This work was done as part of the author's dissertation at the University
of Kansas, under the direction of C.~Huneke. He thanks Huneke and J.~Martin
for helpful discussions, and the referees for their comments. The computer
algebra system \texttt{Macaulay2} by D.~Grayson and M.~Stillman provided
valuable assistance in studying examples.

%\bibliographystyle{amsalpha}
%\bibliography{kummini}

\begin{thebibliography}{Yan00b}

\bibitem[Bar96]{BariNoEqnsCertain96}
Margherita Barile, \emph{On the number of equations defining certain
  varieties}, Manuscripta Math. \textbf{91} (1996), no.~4, 483--494.
  \MR{MR1421287 (97m:13041)}

\bibitem[Bar06]{BariMonomialIdealsNote06}
\bysame, \emph{A note on monomial ideals}, Arch. Math. (Basel) \textbf{87}
  (2006), no.~6, 516--521. \MR{MR2283682 (2007h:13004)}

\bibitem[BH93]{BrHe:CM}
Winfried Bruns and J{\"u}rgen Herzog, \emph{Cohen-{M}acaulay rings}, Cambridge
  Studies in Advanced Mathematics, vol.~39, Cambridge University Press,
  Cambridge, 1993. \MR{95h:13020}

\bibitem[HH05]{HeHiCMbip05}
J{\"u}rgen Herzog and Takayuki Hibi, \emph{Distributive lattices, bipartite
  graphs and {A}lexander duality}, J. Algebraic Combin. \textbf{22} (2005),
  no.~3, 289--302. \MR{MR2181367 (2006h:06004)}

\bibitem[HVT08]{HvThypergraphs08}
Huy~T{\`a}i H{\`a} and Adam Van~Tuyl, \emph{Monomial ideals, edge ideals of
  hypergraphs, and their graded {B}etti numbers}, J. Algebraic Combin.
  \textbf{27} (2008), no.~2, 215--245. \MR{MR2375493 (2009a:05145)}

\bibitem[Kat06]{KatzmanCharIndep06}
Mordechai Katzman, \emph{Characteristic-independence of {B}etti numbers of
  graph ideals}, J. Combin. Theory Ser. A \textbf{113} (2006), no.~3, 435--454.
  \MR{MR2209703 (2007f:13032)}

\bibitem[KTYar]{KTYsmallADeg08}
Kyouko Kimura, Naoki Terai, and Ken-ichi Yoshida, \emph{Arithmetical rank of
  squarefree monomial ideals of small arithmetic degree}, J. Algebraic Combin.
  (To appear).

\bibitem[Lyu88]{LyubArithRk88}
Gennady Lyubeznik, \emph{On the arithmetical rank of monomial ideals}, J.
  Algebra \textbf{112} (1988), no.~1, 86--89. \MR{MR921965 (89b:13020)}

\bibitem[MS05]{MiStCCA05}
Ezra Miller and Bernd Sturmfels, \emph{Combinatorial commutative algebra},
  Graduate Texts in Mathematics, vol. 227, Springer-Verlag, New York, 2005.
  \MR{MR2110098 (2006d:13001)}

\bibitem[Sta97]{StanEC1}
Richard~P. Stanley, \emph{Enumerative combinatorics. {V}ol. 1}, Cambridge
  Studies in Advanced Mathematics, vol.~49, Cambridge University Press,
  Cambridge, 1997, With a foreword by Gian-Carlo Rota, Corrected reprint of the
  1986 original. \MR{MR1442260 (98a:05001)}

\bibitem[SV77]{ScVoSTCI77}
Peter Schenzel and Wolfgang Vogel, \emph{On set-theoretic intersections}, J.
  Algebra \textbf{48} (1977), no.~2, 401--408. \MR{MR0472852 (57 \#12541)}

\bibitem[SV79]{ScmVoSTCI79}
Thomas Schmitt and Wolfgang Vogel, \emph{Note on set-theoretic intersections of
  subvarieties of projective space}, Math. Ann. \textbf{245} (1979), no.~3,
  247--253. \MR{MR553343 (81a:14025)}

\bibitem[Ter99]{Terai99regdual}
Naoki Terai, \emph{Alexander duality theorem and {S}tanley-{R}eisner rings},
  S\=urikaisekikenky\=usho K\=oky\=uroku (1999), no.~1078, 174--184, Free
  resolutions of coordinate rings of projective varieties and related topics
  (Japanese) (Kyoto, 1998). \MR{MR1715588 (2001f:13033)}

\bibitem[Vil01]{VillmonAlg01}
Rafael~H. Villarreal, \emph{Monomial algebras}, Monographs and Textbooks in
  Pure and Applied Mathematics, vol. 238, Marcel Dekker Inc., New York, 2001.
  \MR{MR1800904 (2002c:13001)}

\bibitem[Vil07]{VillUnmBip07}
\bysame, \emph{Unmixed bipartite graphs}, Rev. Colombiana Mat. \textbf{41}
  (2007), no.~2, 393--395.

\bibitem[Wes96]{WestGraphTheory96}
Douglas~B. West, \emph{Introduction to graph theory}, Prentice Hall Inc., Upper
  Saddle River, NJ, 1996. \MR{MR1367739 (96i:05001)}

\bibitem[Yan00a]{YanEtaleGorMcp00}
Zhao Yan, \emph{An \'etale analog of the {G}oresky-{M}ac{P}herson formula for
  subspace arrangements}, J. Pure Appl. Algebra \textbf{146} (2000), no.~3,
  305--318. \MR{MR1742346 (2000k:14041)}

\bibitem[Yan00b]{YanaSRringsDuality00}
Kohji Yanagawa, \emph{Alexander duality for {S}tanley-{R}eisner rings and
  squarefree {$\mathbb N\sp n$}-graded modules}, J. Algebra \textbf{225}
  (2000), no.~2, 630--645. \MR{MR1741555 (2000m:13036)}

\bibitem[Zhe04]{Zhereslnfacets04}
Xinxian Zheng, \emph{Resolutions of facet ideals}, Comm. Algebra \textbf{32}
  (2004), no.~6, 2301--2324. \MR{MR2100472 (2006c:13034)}

\end{thebibliography}

\def\cfudot#1{\ifmmode\setbox7\hbox{$\accent"5E#1$}\else
  \setbox7\hbox{\accent"5E#1}\penalty 10000\relax\fi\raise 1\ht7
  \hbox{\raise.1ex\hbox to 1\wd7{\hss.\hss}}\penalty 10000 \hskip-1\wd7\penalty
  10000\box7}
\providecommand{\bysame}{\leavevmode\hbox to3em{\hrulefill}\thinspace}
\providecommand{\MR}{\relax\ifhmode\unskip\space\fi MR }
\providecommand{\MRhref}[2]{%
  \href{http://www.ams.org/mathscinet-getitem?mr=#1}{#2}
}
\providecommand{\href}[2]{#2}

\end{document}